\documentclass{amsart}
 \newtheorem{thm}{Theorem}[section]

 \newtheorem{prop}[thm]{Proposition}
 \theoremstyle{definition}
 
 \theoremstyle{remark}
 \newtheorem{rem}[thm]{Remark}
 
 \numberwithin{equation}{section}
 \usepackage{amsmath} 
 \usepackage{amssymb}

\title[Estimates of holomorphic functions]{Estimates of holomorphic functions in zero-free domains}
 
\author[A. Borichev] {Alexander Borichev}
\address{Universit\'e Aix-Marseille, 39, rue Joliot Curie, 13453, Marseille Cedex 13, France}           
\email{borichev@cmi.univ-mrs.fr}

\author[V. Petkov] {Vesselin Petkov}
\address{Institut de Math\'ematiques de Bordeaux, 351,
Cours de la Lib\'eration, 33405  Talence, France}
\email{petkov@math.u-bordeaux1.fr}

\thanks{The first author was partially supported by the ANR project DYNOP}


\begin{document}

\newcommand{\1}{{\bold 1}}
\newcommand{\F}{{\mathcal F}}
\newcommand{\CC}{{\mathcal C}}
\newcommand{\K}{{\mathcal K}}
\newcommand{\G}{{\mathcal G}}
\newcommand{\Hh}{{\mathcal H}}
\newcommand{\Z}{{\mathbb Z}}
\newcommand{\Q}{{\mathbb Q}}
\newcommand{\U}{{\mathcal U}}
\newcommand{\A}{{\mathbb A}}
\newcommand{\Ss}{{\mathbb S}}
\newcommand{\C}{{\mathbb C}}
\newcommand{\N}{{\mathbb N}}
\newcommand{\R}{{\mathbb R}}
\newcommand{\D}{{\mathbb D}}

\renewcommand{\Re}{\mathop{\rm Re}\nolimits}
\renewcommand{\Im}{\mathop{\rm Im}\nolimits}

\def\e{\varepsilon}
\def\phi {\varphi}
\def\CC{{\mathcal C}}
\def\ra{\rangle}
\def\la{\langle}
\def\eps{\epsilon}
\def\esp{\vspace{8pt}}
\def\eps{\epsilon}
\def\l2{L^2}
\def\Cp{\C_{+}}
\def\ii{{\bf i}}
\def\Cp{{\mathbb C}_{+}}

\newcommand{\T}{\mathbb{T}\,}

\maketitle

\begin{abstract}
We study functions $f(z)$ holomorphic in $\Cp$ having the property $f(z) \neq 0$ for
$0 < \Im z < 1$  and we obtain a lower bounds for $|f(z)|$ for $0 < \Im z < 1.$ In our analysis we deal with
scalar functions $f(z)$ as well as with operator valued holomorphic functions $I + A(z)$ assuming that $A(z)$ is a trace class operator for $z \in \Cp$ and $I + A(z)$ is invertible for $0 < \Im z < 1$ and is unitary for $z \in \R.$
\end{abstract}

\section{Introduction}

The purpose of this paper is to obtain some estimates on holomorphic functions $f(z)$ in $\Cp$ which have no zeros in
a strip $0 < \Im z < a.$ Our main motivation comes from the scattering theory for the wave equation
in the exterior of a bounded connected domain $K \subset \R^n, \: n\geq 3$, odd, with smooth boundary $\partial K$.
Set $\Omega = \R^n \setminus \bar{K}$ and consider the Dirichlet problem
\begin{equation} \label{eq:1.1}
\begin{cases}(\partial^2_t - \Delta) u =  0\:\: {\rm in} \:\: \R_t \times \Omega,\\
u = 0\:\:{\rm on} \: \R_t \times \partial K,\\
u(0, x) = f_0(x),\: u_t(0, x) = f_1(x). \end{cases}
\end{equation}
The scattering operator $S(\lambda)$ related to (\ref{eq:1.1}) is an operator valued function
$$
S(\lambda): L^2(\Ss^{n-1}) \longrightarrow L^2(\Ss^{n-1}), \quad \lambda \in \R,
$$
which has the form $S(\lambda) = I+ K(\lambda)$ with a trace class operator $K(\lambda)$ (see \cite{LP}). The kernel $a(\lambda, \omega, \theta)$ of $K(\lambda)$ is called the scattering amplitude.

The functions $a(\lambda, \omega, \theta)$ and $S(\lambda)$ are holomorphic for $\Im \lambda \geq 0$ and they admit meromorphic continuation in $\C_{-}$ with
poles $\lambda_j, \: \Im \lambda_j < 0,$ independent of $\omega, \theta.$ For $\Im  \lambda \geq 0$ we have the estimate
$$ |a(\lambda, \theta, \omega)| \le Ce^{\alpha \Im \lambda}(1 + |\lambda|)^M,\: \alpha \geq 0$$
uniformly with respect to $(\omega, \theta) \in S^{n-1} \times S^{n-1}$
 and a similar estimate holds for $\|S(z)\|_{L^2 \to L^2},\: z \in \Cp$. The operator $S(x)$ is unitary for $x \in \R$ and we have the equality
\begin{equation} 
\label{eq:1.4}
S^*(\bar{z}) = S^{-1} (z)
\end{equation}
if $S(z)$ is invertible. This equality shows that the poles of $S(z)$ are conjugated to the points $z \in \Cp$ where $S(z)$ is not invertible. In several important examples there exists a strip
$$
\mathcal U_{\delta} = \{z \in \C_{-}:\: -\delta < \Im z \le 0\},
$$
where $S(z)$ admits an holomorphic extension. For non-trapping obstacles and for some trapping ones related to special geometry of the obstacles we have a polynomial bound on $\|S(z)\|$ for 
$z \in {\mathcal U}_{\delta}$. This bound follows from a bound for the cut-off resolvent $R_{\chi}(z) = \chi(-\Delta - z^2)^{-1}\chi,\: \chi \in C_0^{\infty} (\R^n),\: \chi(x) = 1$ on $K$ in ${\mathcal U}_{\delta}$ (see \cite{V} and \cite{TZ} for non-trapping obstacles and \cite{I} for several strictly convex disjoint obstacles). On the other hand,
these estimates are related to the special geometry of the obstacle and on the properties of the dynamical system connected with the reflecting rays.\\ 
It is an interesting and difficult problem to estimate $\|S(z)\|_{L^2 
 \to L^2}$ for $z \in U_{\delta}$ without any {\bf geometric assumptions} on $K$. An estimate of $S(z)$ for $z \in {\mathcal U}_{\delta}$ implies a similar one for the cut-off resolvent $R_{\chi}(z)$ and this leads to several applications concerning the local energy decay.  In \cite{PS1} the second author and L. Stoyanov proposed the following \\

{\bf Conjecture.} {\em Assume that $S(z)$ has no poles in $U_{\delta}.$ Then for $0 < \delta_1 < \delta$ we have the estimate}
\begin{equation} \label{eq:1.5}
\|S(z)\|_{L^2 \to L^2} \leq C_{\delta_1}e^{c|z|^2},\quad c \ge 0,\: \forall z \in U_{\delta_1}.
\end{equation}
In \cite{BP} this conjecture has been proved for $n = 3$ using a reduction to a semiclassical Schr\"odinger
operator and a suitable estimate for the resolvent of a complex scaling operator. For dimensions $n > 3$ the result in \cite{BP}
seems to be not optimal since we may deduce only a bound
$$
\|S(z)\|_{L^2 \to L^2} \leq Ce^{c|z|^{n-1}},\quad c > 0,\: \forall z \in U_{\delta_1}.
$$
By (\ref{eq:1.4}) the problem is reduced to a upper bound 
$$
\|S^{-1}(z)\|_{L^2 \to L^2} \le e^{c|z|^2},\quad 0 \le \Im z \le \delta_1
$$
which implies an estimate for the adjoint operator $S^*(\bar{z})$.\\

Motivated by the above problem for operator valued holomorphic functions we study scalar holomorphic functions in zero-free domains and we obtain in Section 2
some lower bounds on functions holomorphic in $\Cp$ without zeros in the strip $\{z \in \C: 0 < \Im z < 1\}$.
In Proposition 2.1 we obtain a lower bound for $|f(z)|$ which is very close to an optimal one as we show by an example in Proposition 2.3. For functions $f(z)$ growing as ${\mathcal O}(e^{|z|^{\beta}})$, 
$1<\beta<2$, the result is different and we study this class of  functions in Propositions 2.4 and 2.5. As our examples show, the lower bounds cannot be improved if we have zeros $z_k$ with multiplicities $m(z_k) \to + \infty$. In the physically important examples the resonances and the conjugated zeros are simple (see \cite{KZ}) 
and it is important to search conditions
leading to lower bounds $|f(z)| \ge e^{-a|z|}$ in zero-free domains. 
This problem is treated in Proposition 2.6.\\

In Section 3 we examine the case $I+B(z)$, where $B(z)$ is a finite rank operator valued function holomorphic in $\Cp$ such that $(I + B(z))^{-1}$ exists for $0\le\Im z \le\delta$ and
${\text {\it Image}}\,B(z) \subset V$ with a finite dimensional space $V$ independent of $z$. In particular, we cover the case of matrix valued functions $a(z): \C^m \to \C^m$ holomorphic in $\Cp$ with
$\det a(z) \neq 0$ for $0 \le \Im z \le \delta$. In this generality it seems that this is the first result leading to an estimate on the norm of the inverse matrix and some applications in numerical analysis could be interesting. Next we examine an operator valued function $A(z)$ 
holomorphic in $\Cp$, assuming that $A(z)$ is a trace class operator for $z \in \Cp$ and $I + A(x)$ is unitary for $x\in\R$. We obtain an estimate for $\|(I + A(z))^{-1}\|$ provided that $I + A(z)$ is invertible for $0<\Im z<1$.

\section{Estimates for scalar functions}

In this section we start with the following
\begin{prop} Let $f(z)$ be a holomorphic function in $\Cp$ such that for some $\alpha\ge 0$, $C > 0$, $M\in \N$ we have
$$ 
|f(z)|\le C(1 + |z|)^M e^{\alpha \Im z},\quad z \in \overline{\Cp}.
$$
Assume  that $f(z) \neq 0$ for $0 < \Im z < 1$. Then
\begin{equation} \label{eq:2.1}
\lim_{|x| \to \infty} \frac{\log |f(x + \ii /2)|}{x^2} = 0.
\end{equation}
\end{prop}
\begin{proof}
Consider the function
$$
F(z) = \frac{f(z)e^{i\alpha z}}{C(z + \ii)^M}
$$
which has the same zeros as $f(z)$. Clearly, $F(z)$ is bounded in 
$\Cp$ and we reduce the proof to the case $|f(z)|<1$ for $z\in\Cp$. 
In the strip $\{z\in\C:0<\Im z<1\}$ consider the positive harmonic  function $G(z)=\log(1/|f(z)|)$. Assume that for some $x>1$ we have $G(x + \ii/2)\ge c x^2,\: c>0.$
By Harnack inequality we get
$$
G(x+t+\ii/2)\ge c_1 c x^2,\quad -\frac14\le t\le\frac14,\, c_1>0.
$$

Thus with a constant $c_2>0$ we deduce
$$
\int_{x - 1/4}^{x + 1/4} \log \frac{1}{|f(y + \ii/2)|}\frac{dy}{1 + y^2} 
\ge c_2 c > 0.
$$
If we have
$$
\liminf_{|x|\to\infty} \frac{\log |f(x + \ii/2)|}{x^2} = -c < 0,
$$
then we can find a sequence of points $x_n \in \R$, 
$|x_{n+1}|>|x_n|+1$, $n\ge 0$ so that
$$
\int_{x_n - 1/4}^{x_n + 1/4} \log \frac{1}{|f(y + \ii/2)|}\frac{dy}{1 + y^2} \ge c_2 c > 0
$$
and then
$$
\int_{-\infty}^{+\infty}\log\frac{1}{|f(y+\ii/2)|}\frac{dy}{1+y^2}=+\infty.
$$
This contradicts the standard uniqueness theorem for functions in 
$H^{\infty}(\Cp)$ (see for instance \cite{R}, Chapter 17) and we obtain the result.
\end{proof}

\begin{rem} The assertion of Proposition 2.1 holds for holomorphic functions $f(z)$ in $\Cp$ for which we have $f(z) \neq 0$ for $0 < \Im z < 1$ and
\begin{gather*}
|f(x)|\le C(1+|x|)^M,\quad \forall x \in \R,\\
|f(z)|\le Ce^{\alpha|z|},\quad \alpha \ge 0,\: \forall z \in \overline{\Cp}.
\end{gather*}

\end{rem}
In fact we can consider the function 
$$
F(z) = \frac{f(z) e^{\ii \alpha z}}{(z + \ii)^M}
$$
and apply the Phragm\'en-Lindel\"of principle in the first and the the second quadrant of $\C$ to conclude that $F(z)$ is bounded in $\overline{\Cp}$.\\

To verify that the result of Proposition 2.1 is rather sharp, we establish the following
\begin{prop} Let $\rho(x)$ be a positive function such that 
$\lim_{x \to \infty} \rho(x) = 0$. Then there exists a Blaschke product $B(z)$ in $\Cp$ without zeros in the domain 
$\{z \in \C:\: 0 < \Im z < 1\}$ such that
$$
\liminf_{x \to \infty} \frac{\log |B(x + \ii/2)|}{\rho(x) x^2} < 0.
$$
\end{prop}

\begin{proof} We choose two sequences $x_n\to\infty$, $x_n\ge1$ 
and $k_n \in \N$, $n\ge 1$ so that
\begin{gather} 
\label{eq:2.2}
k_n \ge \rho(x_n) x_n^2,\quad n\ge 1,\\ 
\label{eq:2.3}
\sum_{n \ge1} \frac{k_n}{x_n^2} < \infty.
\end{gather}
Next we set $z_n = x_n + \ii$, $n\ge1$ and consider
$$
B(z)=\prod_{n\ge1}\Bigl(\frac{|z_n^2+1|}{z_n^2+1}\cdot 
\frac{z-z_n}{z-\bar{z}_n}\Bigr)^{k_n}.
$$
The condition (\ref{eq:2.3}) guarantees the convergence of the infinite product. On the other hand, using (\ref{eq:2.2}) we get
$$
|B(x_n+\ii/2)|\le\Bigl|\frac{(x_n+\ii/2)-
(x_n+\ii)}{(x_n+\ii/2)-(x_n-\ii)}\Bigr|^{k_n}=3^{-k_n}<
e^{-\rho(x_n) x_n^2}.
$$
\end{proof}

Now we pass to the analysis of functions $f(z)$ holomorphic in $\Cp$ 
and satisfying the growth condition
\begin{equation} 
\label{eq:2.4}
|f(z)|\le C e^{|z|^{\beta}},\quad 1 <\beta < 2, \: z \in \Cp.
\end{equation}

\begin{prop} Let $1<\beta<2$, and let $f(z)$ be a function holomorphic in $\Cp$ and continuous in $\overline{\Cp}$ satisfying 
$(\ref{eq:2.4})$ and such that $f(x+\ii y)\neq 0$ for $0<y<1$. Then
\begin{equation} 
\label{eq:2.4a}
\liminf _{|x| \to \infty} \frac{\log |f(x + \ii/2)|}{x^{\beta + 1}} > - \infty.
\end{equation}
\end{prop}

\begin{proof} As in the proof of Proposition 2.1, we assume that  
(\ref{eq:2.4a}) does not hold, and obtain that there exists a sequence 
$t_n\to\infty$ such that
\begin{equation} 
\label{eq:2.5}
\log|f(x+\ii/2)|\le-nt_n^{\beta+1},\quad t_n-1/4\le x\le t_n+1/4.
\end{equation}
Now we apply the Carleman formula (see for instance \cite{T}) in the half plane $\Im z \ge1/2$ which yields
\begin{gather*} 
{\mathcal O}(1) \le \frac{1}{\pi R}\int_0^{\pi} \log |f(Re^{i\theta}+\ii /2)| \sin \theta \,d\theta
\\+ \frac{1}{2\pi}\int_1^R \Bigl(\frac{1}{x^2} - \frac{1}{R^2}\Bigr) \log|f(x + \ii/2)f(-x + \ii/2)|\,dx, \qquad R \to \infty.
\end{gather*} 
Therefore, using the notation $\log a=\log^+a-\log^-a$, we obtain
\begin{gather*}
\frac{1}{2\pi}\int_1^R\Bigl(\frac 1{x^2}-\frac{1}{R^2}\Bigr)
\log^-|f(x+\ii/2)f(-x+\ii/2)|\,dx\\ 
\le O(R^{\beta-1})+\frac{1}{2\pi}\int_1^R \frac {\log^+|f(x+\ii/2)f(-x+\ii/2)|}{x^2}\,dx\\
\le O(R^{\beta-1})+C \int_1^R \frac {x^\beta}{x^2}\,dx=
O(R^{\beta-1}), \qquad R \to \infty,
\end{gather*}
and, hence,
$$
\frac{1}{R^2}\int_{R/3}^{2R/3}\log^-|f(x + \ii/2)|\,dx=
O(R^{\beta-1}),\qquad R\to\infty.
$$
This contradicts (\ref{eq:2.5}) for $R = 2 t_n$, $n\to\infty$, 
which completes the proof.
\end{proof}

The following proposition shows how sharp is our lower bound.

\begin{prop} Let $1<\beta<2$, and let $\rho(x)$ be a positive function
such that $\lim_{x\to\infty}\rho(x)=0$. 
Then there exist functions $f$ and $F$ holomorphic in $\Cp$ and continuous in $\overline{\Cp}$ such that $f(x+\ii y)\neq 0$, 
$F(x+\ii y)\neq 0$ for $0<y<1$,
$|f(x+\ii y)|\le c\exp(C y^\beta)$, $x+\ii y\in \overline{\Cp}$, 
$|F(z)|\le c\exp(C |z|^\beta)$, $z\in \overline{\Cp}$, 
$|F(x)|=1$, $x\in\mathbb R$
satisfying the inequalities
$$
\liminf_{x\to+\infty}\frac{\log|f(x+\ii/2)|}{\rho(x)x^{\beta+1}}<0,
\qquad
\liminf_{x\to+\infty}\frac{\log|F(x+\ii/2)|}{\rho(x)x^{\beta+1}}<0.
$$
\end{prop}

\begin{proof} Without loss of generality we assume that
$\lim_{x\to\infty}x\rho(x)=+\infty$.
Given $s \in \mathbb R$, consider the function
$$
B(s,z)=\frac{z-s-\ii}{z-s+\ii}\cdot \frac{s-\ii}{s+\ii}\cdot 
\exp\Bigl[-\frac{2\ii z}{s^2+1}\Bigr].
$$
Then $B(s,x+\ii y)\not=0$, $0<y<1$, $|B(s,x)|=1$, and 
$|B(s,s+\ii/2)|\le c<1/e$ for large $s$.
Next we use two estimates on $B(s,\cdot)$ (see \cite[Chapter 1]{GO}):
\begin{gather}
|\log B(s,z)|=\Bigl| \log\Bigl[\bigl(1-\frac{z}{s+\ii}\bigr)e^{z/(s+\ii)}\Bigr]
-\log\Bigl[\bigl(1-\frac{z}{s-\ii}\bigr)e^{z/(s-\ii)}\Bigr]\Bigr|
\notag\\
=\Bigl| \sum_{m\ge 2}\frac 1m \Bigl[\Bigl(\frac{z}{s+\ii}\Bigr)^m
-\Bigl(\frac{z}{s-\ii}\Bigr)^m\Bigr]
\Bigr|\le C\frac{|z|^2}{s^3},
\qquad |z|\le s/2,\: |s| > 1,
\label{eq:2.7}
\end{gather}
and
\begin{equation}
\log |B(s,z)|\le \Re\frac{-2\ii z}{s^2+1}\le 2\frac{\Im z}{s^2+1}.
\label{xt}
\end{equation}

Choose $t_n\to\infty$, $t_n\ge 1$, and $k_n\ge 1$ such that
\begin{gather}
k_n\ge \rho(t_n)t_n^{\beta+1},\notag\\
\sum_{n\ge 1} \frac{k_n}{t_n^{\beta+1}}<\infty,
\label{eq:2.8}
\end{gather}
and consider
$$F(z)=\prod_{n\ge 1}B^{k_n}(t_n,z).$$
The product converges because of (\ref{eq:2.7}) and (\ref{eq:2.8}). Furthermore, by (\ref{eq:2.7}) and \eqref{xt},
\begin{gather*}
\log|F(z)|\le \sum_{|z|\ge |t_n|/2}k_n\log|B(t_n,z)|+
\sum_{|z|< |t_n|/2}k_n\log|B(t_n,z)| \\
\le C\sum_{|z|\ge |t_n|/2}k_n\frac{|z|}{t^2_n}+
C\sum_{|z|< |t_n|/2}k_n\frac{|z|^2}{t^3_n}.
\end{gather*}
According to (\ref{eq:2.8}), we obtain
$$
\log|F(z)|\le C|z|^\beta.
$$
Finally, for large $n$,
$$
\log|F(t_n+\ii/2)|\le C|t_n|^\beta+k_n\log |B(t_n,t_n+\ii/2)|\le
-\rho(t_n)t_n^{\beta+1}.
$$

Multiplying $F(z)$ by $G(z)=\exp(C\exp(\beta\log (z+i)))$ 
with the branch of the logarithm in the upper half plane 
positive on the imaginary semi-axis
and a suitable $C> 0$, we obtain that the function $f=FG$
satisfies the conditions of our proposition.
\end{proof}

In the above examples the multiplicities of the zeros are not bounded. Motivated by physical examples we would like to examine the situation when the multiplicity of the zeros is bounded, and in addition the zeros satisfy some separation conditions. 
\def\ka{\kappa}
 
\begin{prop} Let $f(z)$ be a function holomorphic in $\Cp$ 
with zeros of bounded multiplicities, such that 
$\log(1/|f(x)|)={\mathcal O}(x)$, $|x| \to\infty$, $x \in \R$.
Assume that for some constants $\alpha\ge 0$, $C>0$, $M\ge 0$ we have
$$ 
|f(z)| \le C(1 + |z|)^M e^{\alpha \Im z},\quad z\in\overline{\Cp}.
$$ 
Moreover, suppose that there exists $k>0$ such that the set of the zeros $\Lambda$ of $f$ in $\Cp$ satisfies the following conditions:
$$
\Im \lambda\ge 1,\qquad \lambda\in\Lambda;
$$
if $\lambda,\mu\in \Lambda$, $\lambda\not=\mu$, and if
$\Im \lambda \le k|\Re \lambda|$, $\Im \mu \le k |\Re \mu|$, then
\begin{equation}
|\lambda-\mu|\ge c(|\lambda|+|\mu|)^{-1/4}. \label{eq:2.9}
\end{equation}
In this situation
$$
-\log|f(x+\ii /2)|= {\mathcal O}(x),\qquad |x| \to \infty.
$$
\end{prop}

\begin{proof} As above we reduce the proof to the case $|f(z)|\le 1$ for $z \in \overline{\Cp}$. From now on, for simplicity, we suppose that $x \ge 1$.

Using the Nevanlinna factorization (\cite{R}, Chapter 17), we represent $f$ as the product
$$
f(z)=e^{\ii a z}B(z)F(z),
$$
where $B$ is the Blaschke product constructed by $\Lambda$,
and $F$ is the outer function determined by the condition 
$|F|=|f|$ on $\R$. Then $|e^{\ii a z}|=e^{-a/2}$, 
$z\in \R+ \ii/2$, and we have 
\begin{gather*}
\log|F(x+\ii /2)|=\frac 1{2\pi}\int_{-\infty}^{\infty}\frac{\log |f(t)|}
{(x-t)^2+1/4}\,dt\\
= \frac 1{2\pi}\int_{-2x}^{2x}\frac{\log |f(t)|}
{(x-t)^2+1/4}\,dt+ \frac 1{2\pi}\int_{\mathbb R\setminus(-2x,2x)}
\frac{\log |f(t)|}{(x-t)^2+1/4}\,dt\\ 
\ge \frac 1{2\pi}\int_{-2x}^{2x}\frac{-cx}
{(x-t)^2+1/4}\,dt- \frac 1{2\pi}\int_{\mathbb R\setminus(-2x,2x)}
\frac{t^2+1}{(x-t)^2+1/4}\cdot\frac{\log^- |f(t)|}{t^2+1}\,dt\\ 
\ge -cx-c_1,\qquad x\to\infty,
\end{gather*}
since
$$
\int_{\R} \frac{\log^-|f(t)|}{t^2+1}\, dt < +\infty,\qquad f \in H^{\infty}(\Cp).
$$
\bigskip

It remains to estimate $|B|$. The Blaschke condition tells us that
$$
\sum_{\lambda\in\Lambda}\frac{\Im \lambda}{1+ |\lambda|^2}\le 
c_0<\infty.
$$
Furthermore,
$$
\log|B(x+\ii/2)|=\sum_{\lambda\in\Lambda} \log
\Bigl|\frac{\lambda-x-\ii/2}{\lambda-x+\ii/2} \Bigr|
=\frac 12\sum_{\lambda\in\Lambda} \log
\Bigl|\frac{\lambda-x-\ii/2}{\lambda-x+\ii/2} \Bigr|^2.
$$
Since $\Im \lambda\ge 1$, $\lambda\in\Lambda$, and 
$\log a\asymp 1-a$, $1/9\le a<1$, we have
\begin{gather*}
\log|B(x+\ii/2)|
  \asymp \sum_{\lambda\in\Lambda}
\Bigl[1-\frac{(x-\Re \lambda)^2+(\Im \lambda-1/2)^2}
{(x-\Re \lambda)^2+(\Im \lambda+1/2)^2} \Bigr]
\\
=\sum_{\lambda\in\Lambda}
\frac{\Im \lambda}
{(x-\Re \lambda)^2+(\Im \lambda+1/2)^2} 
\asymp \sum_{\lambda\in\Lambda}
\frac{\Im \lambda}
{(x-\Re \lambda)^2+(\Im \lambda)^2}.
\end{gather*}

Let $m = 1 + \frac{1}{k}.$ First of all,
$$
\sum_{\lambda\in\Lambda,\, |x-\lambda|\ge x/m}
\frac{\Im \lambda}
{(x-\Re \lambda)^2+(\Im \lambda)^2}\le c_2
\sum_{\lambda\in\Lambda}
\frac{\Im \lambda}
{1+|\lambda|^2}\le c_3.
$$
It remains to estimate the sum 
$$
\sum_{\lambda\in\Lambda_*}
\frac{\Im \lambda}{(x-\Re \lambda)^2+(\Im \lambda)^2}
$$
where
$$
\Lambda_*=\{\lambda\in\Lambda:|x-\lambda|< x/m\} \subset \{\lambda \in \C: 1 \leq \Im \lambda \leq k|\Re \lambda|\}.
$$

For $n\ge 1$ we set $\Lambda_n=\{\lambda\in\Lambda_*:2^{n-1}\le |x-\lambda|< 2^n\}.$
Estimating the area of the domain $\{w:\Im w\ge 0,\, |x-w|< 2^n+1\}$
and using the separation condition \eqref{eq:2.9}, we obtain
\begin{equation} \label{eq:2.10}
{\rm card}\: \Lambda_n\le C_1\cdot 2^{2n}x^{1/2}.
\end{equation}
Furthermore,
$$
(x-\Re \lambda)^2+(\Im \lambda)^2\asymp 2^{2n},\qquad 
\lambda\in\Lambda_n.
$$
We set
$$A_n=\sum_{\lambda\in\Lambda_n} \Im \lambda,\:
B_n=\sum_{\lambda\in\Lambda_n}
\frac{\Im \lambda} {(x-\Re \lambda)^2+(\Im \lambda)^2}\asymp A_n2^{-2n}.$$
Since $\Im \lambda\le 2^n$, $\lambda\in \Lambda_n$, we have
\begin{equation}
A_n\le C_1\cdot 2^{3n}x^{1/2}. \label{eq:2.11}
\end{equation}
Furthermore,
$$
c_0 \ge\sum_{\lambda\in\Lambda_*}\frac{\Im \lambda}{1+|\lambda|^2}\ge 
\frac{c_4}{1+x^2}\sum_{\lambda\in\Lambda_*}\Im \lambda,
$$
and, hence,
\begin{equation} \label{eq:2.12}
\sum_{n\ge 1}A_n\le C_2(1+x^2).
\end{equation}
Finally, we obtain that
\begin{gather*}
\sum_{\lambda\in\Lambda_*}
\frac{\Im \lambda}
{(x-\Re \lambda)^2+(\Im \lambda)^2}=
\sum_{n\ge 1}B_n\asymp 
\sum_{n\ge 1}A_n2^{-2n}\\=
\sum_{2^n<x^{1/2}}A_n2^{-2n}+\sum_{2^n\ge x^{1/2}}A_n2^{-2n}.
\end{gather*}
By \eqref{eq:2.11} and \eqref{eq:2.12} we conclude that the right hand part is estimated by
$$
 c_5\sum_{2^n<x^{1/2}}2^nx^{1/2}+\frac{1}{x}\sum_{2^n\ge x^{1/2}}A_n\le C_3x,\qquad x\ge 1.
$$
\end{proof}

\begin{rem} The restriction on the multiplicity of the zeros of $f$ is fulfilled in many physical examples since we know that for generic perturbations the resonances are simple (see \cite{KZ}).
\end{rem}

\begin{rem} The separation condition is used only in the estimation of the number of zeros belonging to $\Lambda_n$. Thus our argument works assuming only that \eqref{eq:2.10} holds without any restriction on the multiplicity of the zeros in 
$\{\lambda \in \C:\: \Im \lambda \le k|\Re \lambda|\}$. Moreover, 
we can improve the lower order bound of $|f(x + \ii/2)|$ if we have a stronger separation condition
$$
|\lambda-\mu|\ge d > 0,\quad 
\lambda,\mu\in \Lambda,\,\lambda\not=\mu,\,
\Im \lambda \le k|\Re \lambda|,\, \Im \mu \le k |\Re \mu|.
$$
We refer to \cite{TZ} for examples and comments concerning  separation conditions on the resonances.

\end{rem}
\section{Estimates for $(I + B(z))^{-1}$}

Let $H$ be a Hilbert space with scalar product $(\cdot,\cdot)$ and norm $\|\cdot\|$. We denote also by $\|\cdot\|$ the norms of operators in $H$ and by ${\mathcal L}(H)$ the space of bounded linear operators on $H$. Let 
$B(z): z \in \Cp \to {\mathcal L}(H)$
be an operator valued function. We will prove the following

\begin{thm} Let $B(z)$ be holomorphic in $\Cp$ and such that for some constants $\alpha\ge 0$, $C > 0$, $M \ge 0$ we have
$$
\|B(z)\| \le C(1 + |z|)^M e^{\alpha \Im z},\quad z \in \overline{\Cp}.
$$
Assume that $(I + B(z))^{-1} \in {\mathcal L}(H)$ for $0 < \Im z < 1$ and let ${\text {\it Image}}\,B(z) \subset V$, $V$ being a finite dimensional space of $H$ independent of $z\in \overline{\Cp}$. Then for every $\epsilon > 0$ we have
$$ \|(I + B(x + \ii/2))^{-1}\| \leq C_{\epsilon} e^{\epsilon |x|^2}.$$
\end{thm}

\begin{proof} Choose an orthonormal basis $\{e_1,...,e_N\}$ in $V$ and let $H = V \oplus V^{\perp}$. Given $g \in H$ we write 
$g = g_1 + g_2$, $g_1 \in V$, $g_2 \in V^{\perp}$ and consider the equation
\begin{equation} 
\label{eq:3.1}
(I + B(z))f(z) = g_1 + g_2.
\end{equation}
Setting $f(z) = f_1(z) + f_2(z)$, with $f_1(z) \in V$, 
$f_2(z) \in V^{\perp}$, we get $f_2(z) = g_2$ and we  reduce (\ref{eq:3.1}) to
$$
f_1(z) + B(z)f_1(z) = -B(z)g_2 + g_1 = h(z).
$$
Next we have $B(z) e_j = \sum_{k=1}^N (B(z)e_j, e_k) e_k,\: j =1,\ldots,N $
and we search $f_1(z)$ in the form $f_1(z) = \sum_{k = 1}^N c_k(z) e_k.$ For the functions $c_k(z)$ we get a linear
system
$$
c_k(z) + \sum_{j =1}^N c_j(z) (B(z) e_j, e_k) = (h(z), e_k) = h_k(z),
\quad k = 1,\ldots,N.
$$
Introduce the $(N \times N)$ matrix $A(z)$ with the elements 
$a_{i,j}(z)= (B(z)e_j, e_i)$, $i,j=1,\ldots,N$. Then we must
solve the equation
$$
(I+A(z))u(z)=w(z), \quad 
u(z)=\begin{pmatrix}
c_1(z)\\
\ldots\\
c_N(z)
\end{pmatrix}, \quad 
w(z) = \begin{pmatrix}
h_1(z)\\
\ldots\\
h_N(z)
\end{pmatrix}.
$$
Our hypothesis shows that $a(z) = \det(I + A(z)) \neq 0$ for 
$0 < \Im z < 1$. Moreover, $a(z)$ is holomorphic in $\Cp$ and 
$$
|a(z)|\le C_N(1 + |z|)^{MN} e^{\alpha N \Im z},\quad z \in \overline{\Cp}.
$$
Furthermore, we have $ u(z) = \frac{1}{a(z)} D(z) w(z)$
with a matrix $D(z)$ such that 
$$
\|D(z)w(z)\| \le C_N'(1 + |z|)^{NM} e^{\alpha N\Im z}\|g\|,\quad z \in \overline{\Cp}.
$$
Therefore,
$$
\|(I + B(x + \ii/2))^{-1}\| \le \frac{1}{|a(x + \ii/2)|} D_N(1 + |x|)^{NM} 
e^{\alpha N/2},
$$

An application of Proposition 2.1 yields a lower bound of $|a(x + \ii /2)|$
and the proof is complete.
\end{proof}

\def\t1{\mathcal T_1}
Now consider an operator valued holomorphic function
$$
A(z):\: z\in \Cp \to \t1
$$
where $\t1$ denotes the space of trace class operators in $H$ 
with the norm $\|\cdot\|_1$. Recall that for every $B \in \t1$ we can define the determinant 
$$
\det(I + B) = \prod_{j} (1 + \lambda_j(B)),
$$
$\lambda_j(B)$ being the eigenvalues of $B$ and
$$ 
|\det(I + B)| \le e^{\|B\|_1}.
$$
Moreover, given $B \in \t1$ we may consider the function
$$
F_B(\mu) = \bigl[\det(I + \mu B)\bigr] (I + \mu B)^{-1}
$$
which extends from the set $\{\mu\in\C: -\mu^{-1} \notin \sigma(B)\}$ 
to an entire operator valued function in $\C$ such that
$$
|F_B(\mu)| \le e^{\|B\|_1 |\mu|},\quad \mu \in \C.
$$
We refer to \cite[Chapter XIII, Section 17]{RS} 
for the above mentioned properties. 

Next, if $A(z)$ is holomorphic in $\Cp$, then the function 
$\det(I + A(z))$ is also holomorphic in $\Cp$ (see for instance 
\cite{RS}), and if $I+A(z_0)$ is invertible, then we have 
$\det(I + A(z_0)) \neq 0$. 
An application of Proposition 2.1 leads to the following

\begin{thm} Let $A(z)$ be a holomorphic function in $\Cp$ with values in $\t1$ such that 
$$
\|A(z)\|_1 \le C( 1 +|z|),\quad z \in \Cp.
$$ 
Assume that $I + A(x)$ is unitary for $x\in \R$ and suppose that for 
$0 < \Im z < 1$ the operator $I + A(z)$ is invertible. Then for every 
$\varepsilon>0$ we have
\begin{equation} 
\label{eq:3.3}
\|(I + A(x + \ii/2))^{-1}\| \le C_{\varepsilon} e^{\varepsilon |x|^2}.
\end{equation}
\end{thm}

\begin{proof} We have $|\det(I + A(x))| = 1$ for $x \in \R$, 
$\det(I + A(z)) \neq 0$ for $0 < \Im z < 1$ and
$$
|\det(I + A(z))|\leq C_1 e^{C|z|},\quad z \in \Cp.
$$
By Proposition 2.1 and Remark 2.2, we obtain a lower bound for  
$|\det(I + A(x + \ii/2))|$. Combining this bound with the estimate
$$
\|F_{A(z)}(1)\| \le C_1 e^{C|z|}, \quad z \in \Cp,
$$ 
we obtain (\ref{eq:3.3}).
\end{proof}

Proposition 2.6 together with the above argument gives us immediately the following

\begin{thm} Under the assumptions of Theorem $3.2$ suppose that the points $z \in \Cp$ for which $I + A(z)$ is not invertible
satisfy the separation condition \eqref{eq:2.9}. Then we have the estimate
\begin{equation} \label{eq:3.4}
\|(I + A(x + \ii/2))^{-1}\| \leq C e^{c|x|}, \quad x \in \R.
\end{equation}
\end{thm}

\begin{rem} In the scattering theory the scattering operator 
$S(z)=I+K(x)$ is unitary for $x\in\R$, and the scattering determinant 
$a(z)=\det(I+K(z))$ is holomorphic in $\Cp$. Finding an estimate for 
$a(z)$ in $\Cp$ is rather complicated. It was proved in \cite{PZ} that we have
$$
|a(z)|\le C_1 e^{\alpha |z|^{n-1} \Im z},\quad \alpha \ge 0, z \in \Cp.
$$
\end{rem}

Thus it is interesting to examine the estimates of holomorphic functions $f(z)$ in $\Cp$ growing like
$$
|f(z)|\le e^{\alpha |z|^{\gamma}}, \quad \gamma > 1, \forall z \in \Cp.
$$
In this direction the results in Section 2 show that without some additional conditions on $f(z)$ we cannot expect to obtain lower bounds on $|f(z)|$ for $0 < \Im z < 1$ better than those obtained in Proposition 2.4.

\end{document}